\documentclass[11pt]{amsart}     
\usepackage{amssymb,amsthm,graphicx,here,longtable}
\newtheorem{theorem}{Theorem}

\newtheorem{proposition}{Proposition}
\newtheorem{lemma}{Lemma}

\newtheorem{remark}{Remark}

\newcommand{\ord}{{\rm{ord}}}
\newcommand{\sing}{{\rm{Sing}}}
\begin{document}
\title[On $(2,3)$ torus decompositions]
{On $(2,3)$ torus decompositions of $QL$-configurations}
\author[M. Kawashima and K. Yoshizaki]
{Masayuki Kawashima and Kenta Yoshizaki}


\address{\vtop{
\hbox{Department of Mathematics}
\hbox{Tokyo University of science}
\hbox{wakamiya-cho 26, shinjuku-ku}
\hbox{Tokyo 162-0827 Japan}
\hbox{\rm{E-mail}: {\rm kawashima@ma.kagu.tus.ac.jp}}
}}

\maketitle

\begin{abstract}
 Let $Q$ be an affine quartic which does not 
 intersect transversely  with the line at infinity $L_{\infty}$.
In this paper,
 we show the existence of a $(2,3)$ torus decomposition of 
 the defining polynomial of $Q$ and its uniqueness except for one class.
\end{abstract}

\section*{Introduction}\label{s0}

Let $C=\{f=0\}\subset \Bbb{C}^2$
be an irreducible affine plane curve and
let $a, b$ be coprime positive integers with $a,b\ge 2$.
We say that $C$ is
   {\em a  quasi torus curve of type $(a,b)$} (c.f \cite{Ku-Albanese})
    if there exist polynomials
   $f_r$, $f_p$ and $f_q$ such that
   they satisfy the following condition: 
 \[
  \tag{$*$} f_r(x,y)^{ab}f(x,y)=f_p(x,y)^a+f_q(x,y)^b,\quad
  \deg f_j = j,\ \ j=r,p,q
 \]
   where $r\ge 0$ and $p$, $q>0$.
   Under using affine equations,
   we also add conditions that 
 any two polynomials of $f$, $f_r$, $f_p$ and $f_q$ are coprime.
 We say that such a decomposition $(*)$ is
  {\em a quasi torus decomposition of $C$}. 
A quasi torus curve $C$ is called {\em torus curve}
  if $f_r(x,y)$ is a non-zero constant.

In \cite{Tokunaga-Kyoto},
H. Tokunaga studied $D_{2p}$ covers of $\Bbb{P}^2$
branched along a quintic $Q+L_{\infty}$
where $Q$ is a quartic and  $L_{\infty}$ is the line at infinity and
their relative positions are the following:
 \begin{itemize}
  \item   $Q\cap L_{\infty}$ consists of two points.
\begin{enumerate}
  \item[(i)]  $L_{\infty}$ is bi-tangent to $Q$ at two distinct smooth points. 
\item[(ii)] $L_{\infty}$ is tangent to a smooth point and passes through
              a singular point of $Q$. 
\item[(iii)] $L_{\infty}$ passes through two distinct singular points of
             $Q$.
   \end{enumerate}
  \item $Q\cap L_{\infty}$  consists of a point.
\begin{enumerate}
  \item[(iv)] $L_{\infty}$ is tangent to $Q$ at a smooth point with intersection multiplicity $4$. 
\item[(v)] $L_{\infty}$ intersects $Q$ at a singular point with intersection multiplicity $4$.
\end{enumerate}
       \end{itemize}
Table 1 is the list of the possible configurations
$Q+L_{\infty}$ which is given in \cite{Tokunaga-Kyoto}.
\begin{center}
\begin{tabular}{|c|c|c|c|c|c|c|c|}
\hline
No. & \quad $\sing(Q)$\quad   \quad & $Q\cap L_{\infty}$ & No. &\quad
 $\sing(Q)$\quad \  & $Q\cap L_{\infty}$ \\ \hline
(1) & $2a_2$ & (i) & (11) & $e_6$ & (i) \\
(2) & $2a_2$ & (iv) &  (12) & $e_6$ & (iv)    \\
(3) & $2a_2+a_1$ & (i) &   (13) & $a_4+a_2^{\infty}$ & (ii) \\
(4) & $2a_2+a_1$ &\ (iv) & (14) & $a_3^{\infty}+a_2+a_1$ & (ii) \\
(5) & $3a_2$ & (i) & (15) & $a_5+a_1$ & (i) \\
(6) & $a_2+a_3^{\infty}$ & (ii)
                  & (16) & $a_5^{\infty}+a_2^{\infty}$ & (iii) \\
(7) & $a_5$ & (i) & (17) & $2a_3^{\infty}$ & (iii) \\
(8) & $a_5$ & (iv) & (18) & $a_7^{\infty}$ & (v) \\
(9) & $a_6^{\infty}$ & (ii) & (19) & $2a_3^{\infty}+a_1$ & (iii) \\
(10) & $a_4^{\infty}+a_2$ & (v)  & & & \\ \hline
\end{tabular}\\[0.3cm]
Table 1. \\[1mm]
\end{center}
where the singularities of types $a_n$ and $e_6$ are defined by
\[
 a_n:x^2+y^{n+1}=0\  (n\ge 1),\quad e_6:x^3+y^4=0
\]
 and the notation  $*^\infty$ express singularities on the line
 $L_{\infty}$.
We call such configurations {\it $QL$-configurations}.
Note that $Q$ is irreducible  for the cases $(1),\dots, (13)$ and 
$Q$ is not irreducible  for the cases $(14),\dots, (19)$.
We call configurations for the cases $(1),\dots, (13)$
(respectively for the cases $(14),\dots, (19)$)
{\it irreducible $QL$-configurations}
(resp. {\it non-irreducible $QL$-configurations}). 

In \cite{Y}, the second author studied two topological invariants of
\linebreak
 $QL$-configurations,
 the fundamental group $\pi_1(\Bbb{C}^2\setminus Q)$ and
 the Alexander polynomial $\Delta_Q(t)$.
 In particular, he showed that
\[
 \tag{$\star$}
 {\text{$t^2-t+1$ divides $\Delta_Q(t)$ except
 for the case $(13)$ and $(16)$. }}
\]

In \cite{okatan},
 M. Oka studied a special type of degeneration family
 $\{C_\tau\}$ of irreducible torus sextics
 which degenerates into $C_0:=D+2L_{\infty}$
 where $D$ is a quartic and $L_{\infty}$ is a line.
 We call such a degeneration {\em{a line degeneration of order $2$}}.
 (We will give the definition for general situation in \S1).
 He showed the divisibility of the Alexander polynomials
 $\Delta_{C_{\tau}}(t)\,|\, \Delta_{D}(t)$ for $\tau \ne 0$.
  (Theorem 14 of  \cite{okatan})

 In this paper, we study the possibilities of quasi torus
decompositions of $QL$-configurations so that
the above divisibility $(\star)$
also follows from  the results of the line degeneration by M. Oka.

 This paper consists of 9 sections.
In section 1,
we recall the definition of
line degenerated torus curves of type $(p,q)$.
In section 2,
 we classify the singularities of
 line degenerated torus curves of type $(2,3)$.
In section 3, we state our main theorem.
In sections 4, 5, 6 and 7, we prove the theorems which are
stated in \S 3
using  line degenerated torus curves of type $(2,3)$.

 For the cases $(14),\dots, (19)$,
  $Q$ is not irreducible but
 we can also consider torus decomposition of non-irreducible
 $QL$-configurations
 without irreducibility and the condition $\gcd(a,b)=1$
 in the definition of quasi torus curves. 
 In section 8,
we consider torus decomposition of above non-irreducible
 $QL$-configurations. 
In section 9, we will show that
if a plane curve  has a $(2,3)$ torus decomposition,
then there exist infinite $(2,3)$ quasi torus decompositions.

\section{Line degenerated torus curves}\label{s1}
Let $U$ be an open neighborhood of $0$ in $\Bbb{C}$
and let $\{C_s\,|\,s\in U\}$ be an analytic family of
  irreducible curves of degree $d$
which 
  degenerates into $C_0:=D+j\,L_{\infty}$ $(1\le j < d)$
  where $D$ is an irreducible curve of degree $d-j$ and
  $L_{\infty}$ is a line.
We assume that
  there is a point $B\in L_{\infty}\setminus L_{\infty}\cap D$ such
  that $B\in C_s$ and the multiplicity of $C_s$ at $P$ is $j$
  for any non-zero $s \in U$.
We call such a degeneration {\em a line degeneration of order $j$} and
  we call $L_{\infty}$ {\em the limit line} of the degeneration.
 $B$ is called {\em the base point} of the degeneration.
In \cite{okatan}, M. Oka showed that
there exists a canonical surjection:
\[
 \varphi:\pi_1(\Bbb{C}^2\setminus D)\to
 \pi_1(\Bbb{C}^2\setminus C_s),\quad
 s: {\text{sufficiently small,}} 
\]
where $\Bbb{C}^2=\Bbb{P}^2\setminus L_{\infty}$
 and as a corollary he showed  the divisibility among the Alexander polynomials of
a line degeneration family:
\[
 \Delta_{C_{s}}(t)\mid \Delta_{D_0}(t). 
\]

\subsection{Line degenerated torus curves} 
Let $C_{p,q}=\{F_{p,q}=0\}$ be
a $(p,q)$ torus curve $(p > q\ge 2)$  where $F_{p,q}$ is defined by 
\[\tag{$1$}
  F_{p,q}(X,Y,Z)=F_p(X,Y,Z)^q+F_q(X,Y,Z)^p,\quad
  \deg F_k=k,\  k=p,q
\]
  where $(X,Y,Z)$ is a homogeneous coordinates system of $\Bbb{P}^2$.
Suppose that $F_{p,q}$ is written by the following form:
\[\tag{$2$}
  F_{p,q}(X,Y,Z)=Z^j\,G(X,Y,Z)
\]
  where $G(X,Y,Z)$ is a homogeneous polynomial of degree $pq-j$.
We call a curve $D=\{G=0\}$
  {\em a line degenerated torus curve of type $(p,q)$ of order $j$} and
  the line $L_{\infty}=\{Z=0\}$
  {\em the limit line of the degeneration}.  
We divide the situations $(2)$ into two cases.

\noindent
{\bf{First case.}}
Suppose that the defining polynomials of associated curves
are written as follows: 
\[
 F_p(X,Y,Z)=
 F_{p-r}'(X,Y,Z)Z^r,\quad
 F_q(X,Y,Z)=F_{q-s}'(X,Y,Z)Z^s
\]
where $r$ and $s$ are positive integers such that $r < p$ and $s < q$.
We assume that $sp\ge rq$.
Factoring $F_{p,q}$ as
$F_{p,q}(X,Y,Z)=Z^{rq}G(X,Y,Z)$,
 we can see that $G$ is defined as
\[
\tag{$3$}
 G(X,Y,Z)=F_{p-r}'(X,Y,Z)^q+F_{q-s}'(X,Y,Z)^pZ^{sp-rq}.
\]
We call such a factorization {\em visible factorization} and
$D$ is called {\em a visible degeneration of torus curve of type $(p,q)$}.
By the definition,
$D\cap L_{\infty}=\{F_{p-r}'(X,Y,Z)=Z=0\}$ and thus
the limit line $L_{\infty}$ is singular with respect to
the visible degeneration of torus curve $D$.
In \cite{okatan}, M. Oka showed that
a visible 
degeneration of torus curve of type $(p,q)$
can be expressed as a line degeneration of irreducible
torus curves of degree $pq$.

\noindent
{\bf{Second  case.}}
 Neither $F_p$ or $F_q$ factors through $Z$
 but $F$ can be written as $(2)$.
Then $D$ is called  
 {\em an invisible 
 degeneration of torus curve of type $(p,q)$}.

\section{Line degenerated $(2,3)$\label{s2}
         torus curves of degree $4$}
In this section, we consider a $(2,3)$ sextic of torus type
which is a visible factorization.

\subsection{Visible factorization}
 Let $D=\{G=0\}$ be a quartic
 associated with  a visible factorization $(3)$:
\[
 D:\quad   G(X,Y,Z)=F_{2}'(X,Y,Z)^2+F_{1}'(X,Y,Z)^3Z=0.
\]
 We consider the associated curves $C_2:=\{F_2'=0\}$,
      $L:=\{F_1'=0\}$ and $L_{\infty}:=\{Z=0\}$.
 Let $P$ be an inner  singularity of $D$,
 namely $P$ is on the intersection
 $C_2\cap L$,  $C_2\cap L_{\infty}$ or $C_2\cap L\cap  L_{\infty}$.
 Then the topological type $(D,P)$  depends only on
 the intersection multiplicities  of $C_2$, $L$ and $L_{\infty}$.
To describe singularities of $D$, we put
the intersection multiplicities
     $\iota_1:=I(C_2,L;P)$ and
     $\iota_2:=I(C_2,L_{\infty};P)$.
Note that $0\le \iota_i \le 2$ for $i=1,2$ and
$(\iota_1,\iota_2)\ne (2,2)$ as $L\ne L_{\infty}$.
     
\begin{lemma}\label{l1}
 Suppose that $C_2$ is smooth at $P$.
 Then we have the following descriptions:
  \begin{enumerate}
  \item[{\rm{(1)}}]  If $P\in C_2\cap L\setminus L_{\infty}$, 
         then we have $(D,P)\sim a_{3\iota_1-1}$.
         
  \item[{\rm{(2)}}]  If $P\in C_2\cap L_{\infty}\setminus L$, 
         then $D$ is smooth at $P$ and is tangent to $L_{\infty}$
         with $I(D,L_{\infty};P)=2\iota_2$.
         
  \item[{\rm{(3)}}]  If $P\in C_2\cap L\cap  L_{\infty}$, then we have
         $(D,P)\sim a_{3\iota_1+\iota_2-1}$.
\end{enumerate}
\end{lemma}
\begin{proof}
The assertion $(1)$ is showed in \cite{Pho} and
 \cite{BenoitTu} 
 for general cases. 
We consider the case (2) and (3).
 We use affine coordinates $(x,z)=(X/Y,Z/Y)$
 on $\Bbb{C}^2=\Bbb{P}^2\setminus \{Y=0\}$.
 Then the defining polynomial $g$ of $D$ in the affine coordinates
 is given as follows.
 \[
 g(x,z):=G(x,1,z)=f_2'(x,z)^2+f_1'(x,z)^3z,\quad
 f_j'(x,z):=F_j'(x,1,z),\,j=1,2.
\]
As $C_2$ is smooth at $P$,
 we can take  local coordinates $(u,v)$ so that
 $L_{\infty}=\{v=0\}$ and
 $f'_2(u,v)=c_1\,v-\varphi(u)$ where $c_1\ne 0$.
 Then 
 $\ord_u \varphi(u)=\iota_2$ and 
 \[
 \begin{split}
   g(u,v)&=(c_1 v-\varphi(u))^2+f'_1(u,v)^3v\\
         &=c\,u^{2\iota_2}+f'_1(P)^3\,v+{\text{(higher terms)}},
 \ c\ne 0.
\end{split}
 \]
Here ``higher terms'' are linear combinations of monomials
 $u^{\alpha}v^{\beta}$
 such that $2\iota_2\beta+\alpha > 2\iota_2$.
 They do not affect the topology of $D$ at $P$. 
 As $P$ is not on  $L$, $f'_1(P)$ is not $0$.
 Hence we have the assertion (2).

 \newpage
 Now we show the assertion (3).
 As $C_2$ is smooth at $P$,
 we can take 
 local coordinates $(u,v)$ so that
 $f'_2(u,v)=v$, $f'_1(v,v)=c_1\,v-\varphi_1(u)$ and
 $L_{\infty}=\{c_2\,v-\varphi_2(u)=0\}$
 where $c_1$ and $c_2$ are non-zero constant.
 Then  
 $\ord_u \varphi_1(u)=\iota_1$ and 
 $\ord_u \varphi_2(u)=\iota_2$ and
 \[\begin{split}
   g(u,v)&=v^2+(c_1\,v-\varphi_1(u))^3(c_2\, v-\varphi_2(u))\\
   &=v^2-c\,u^{3\iota_1+\iota_2}+{\text{(higher terms)}},\ 
   c\ne 0.
\end{split}
 \]
Here ``higher terms'' are linear combinations of monomials
 $u^{\alpha}v^{\beta}$
 such that $2\alpha+(3\iota_1+\iota_2)\beta > 2(3\iota_1+\iota_2)$
 if $\gcd(2,3\iota_1+\iota_2)=1$ and
 $\alpha+(3\iota_1+\iota_2)\beta/2 > (3\iota_1+\iota_2)$
 if $\gcd(2,3\iota_1+\iota_2)=2$.
In particular, $v\varphi_1(u)^3$ is in  {\rm{(higher terms)}}.
 This shows the assertion (3).
\end{proof}

Next we consider the case that $C_2$ is singular at $P$.
Then $C_2$ consists of two lines $\ell_1$ and $\ell_2$ such that
$\ell_1\cap \ell_2=\{P\}$.

\begin{lemma}\label{l2}
  Suppose that $C_2$ is singular at $P$.
  Then singularities of $D$ at $P$ are described as follows:
    \begin{enumerate}
      \item[{\rm{(1)}}]  If $P\in C_2 \cap L\setminus L_{\infty}$,
             then we have $(D,P)\sim e_6$.
      \item[{\rm{(2)}}]  If $P\in C_2 \cap L_{\infty} \setminus L$,
             then $D$ is smooth at $P$ and is tangent to $L_{\infty}$
             with $I(D,L_{\infty};P)=4$.
      \item[{\rm{(3)}}]  If $P\in C_2 \cap L \cap L_{\infty}$,
             then $D$ consists of four lines
             which intersect at $P$.      
    \end{enumerate}
\end{lemma}
\begin{proof}
The assertion $(1)$ is showed in \cite{Pho}.
 We show the assertion (2).
 We take a suitable local coordinates $(u,v)$ at $P$ so that
 $f_2(u,v)=u(b_1u-b_2v)$ and $L_{\infty}=\{v=0\}$
 where $b_i\ne 0$ for $i=1,2$.
 Then
 \[\begin{split}
  g(u,v)&=u^2(b_1u-b_2v)^2+b_3\,f'_1(P)^3v\\
        &=b_1^2u^4+b_3\,f'_1(P)^3v+{\text{(higher terms)}}.
    \end{split}
 \]
 As $P$ in not on  $L$, we have $f'_1(P)\ne 0$ and
 $I(D,L_{\infty};P)=4$.
 This shows the assertion (2).
The assertion $(3)$ is obvious form the defining polynomial of $D$.
\end{proof}

We say that $P$ is {\em an outer singularity of $D$}
if $P\in \sing(D)\setminus C_2$.
We consider possible outer singularities of $D$.

\begin{lemma}\label{l3}
  If $P\in D$ is an outer singularity,
            then $(D,P)$ is either $a_1$ or $a_{2}$.
 \end{lemma}
Our proof is computational and it is done in the same way as in
 \cite{Oka-Pho2}.

\subsection{Invisible factorization}
Let $D=\{G=0\}$ be an invisible factorization of a $(2,3)$ torus curve
which satisfies the following equations.
 \[
 \tag{$2$}
  F_{2,3}(X,Y,Z)=
  F_2(X,Y,Z)^3-F_3(X,Y,Z)^2=Z^2\,G(X,Y,Z)
\]
We assume that $G$ is not divided by $Z$ and $G$ is reduced.
We rewrite $F_2$ and $F_3$ as follows:
\[
\begin{split}
 F_2(X,Y,Z)&=F_2^{(2)}(X,Y)+F_2^{(1)}(X,Y)Z+F_2^{(0)}(X,Y)Z^2,\\
 F_3(X,Y,Z)&=F_3^{(3)}(X,Y)+F_3^{(2)}(X,Y)Z+F_3^{(1)}(X,Y)Z^2
 +F_3^{(0)}(X,Y)Z^3
\end{split}
\]
where $F_j^{(i)}$ is a homogeneous polynomial of degree $i$.
 By an easy calculation, we observe that there exists a linear form 
 $\ell_1(X,Y)$ so that 
 \[
   \begin{cases}
   F_2^{(2)}(X,Y)=\ell_1(X,Y)^2,\\
   F_3^{(3)}(X,Y)=\varepsilon\,\ell_1(X,Y)^3,\quad
   F_3^{(2)}(X,Y)=\dfrac{3\varepsilon}{2}\ell_1(X,Y)F_2^{(1)}(X,Y)
  \end{cases}
\]
 where $\varepsilon=1$  or $-1$.
  We put $\ell_2(X,Y):=F_2^{(1)}(X,Y)$ and
 $\ell_3(X,Y):=F_3^{(1)}(X,Y)$.
 Then we may assume  the defining polynomials of $C_2$ and $C_3$
as the following:
\begin{eqnarray*}(\sharp)\quad
 \begin{cases}
\begin{split}
 F_2(X,Y,Z)&=\ell_1(X,Y)^2+\ell_2(X,Y)\,Z+a_{00}\,Z^2,\\ 
 F_3(X,Y,Z)&=\ell_1(X,Y)^3+\frac{3}{2}\,\ell_1(X,Y)\,\ell_2(X,Y)\,Z
                 +\ell_3(X,Y)\,Z^2+b_{00}\,Z^3.
 \end{split}
\end{cases}
\end{eqnarray*}
Then $F_{2,3}$ is factorized  as
\[
\begin{split}
 F_{2,3}(X,Y,Z)=F_2(X,Y,Z)^3-F_3(X,Y,Z)^2=Z^2G(X,Y,Z). 
\end{split}
\]
To see the local geometry of $D$ at a intersection point 
$D$ and $L_{\infty}$,
 we may assume  $\ell_1(X,Y)=X$ and
we take the affine coordinates
 $(x,z)=(X/Y,Z/Y)$ at $O^*:=[0:1:0]$.
 Let $g(x,z)=G(x,1,z)$, $f_2(x,z)=F_2(x,1,z)$ and
  $f_3(x,z)=F_3(x,1,z)$ 
 be the local equations of $D$, $C_2$ and $C_3$ respectively.
 In the affine  coordinates $(x,z)$,
 $f_2$ and $f_3$ are written as
\[
\begin{split}
 f_2(x,z)&=x^2+\ell_2(x,1)\,z+a_{00}z^2,\\
 f_3(x,z)&=x^3+\frac{3}{2}\ell_2(x,1)\,xz+\ell_3(x,1)z^2+b_{00}z^3.
\end{split}
\]
We can see the local geometries of $C_2$ and $C_3$ at $O^*$.
 First we consider the case $\ell_2(0,1)\ne 0$.
 Then we have 
\begin{enumerate}
\item[{\rm{(1)}}] $C_2$ is smooth at $O^*$ and is tangent to
      the limit line $L_{\infty}$ at $O^*$.
\item[{\rm{(2)}}] $C_3$ has an $a_1$ singularity at $O^*$.
\item[{\rm{(3)}}] The intersection multiplicity $I(C_2,C_3;O^*)$ is $3$.
\end{enumerate}
Then, putting $c_1=\ell_2(0,1)$, $g(x,z)$ is given as 
\[
 g(x,z)=c_1^3\,z+\frac{3}{4}\,c_1^2\,x^2+{\text{(higher terms)}}.
\]
Thus $D$ is simply tangent to $L_{\infty}$ at $O^*$.
We write  $g(x,0)=x^2\,(x-\alpha)(x-\beta)$ for some $\alpha,\beta$
 such that $\alpha \beta=3c_1^2/4\ne 0$.
Then if $\alpha\ne \beta$,
        then $L_{\infty}$ is tangent to $D$ at $O^*$ and
        intersects transversely with $D$ at other $2$ points.
     If $\alpha=\beta$, then
        $L_{\infty}$ is a bi-tangent line of $D$.

\begin{lemma}\label{l4}
If $c_1\ne 0$, 
 then the set of singularities $\sing(D)$ is
$\{3a_2\}$ or $\{a_2+a_5\}$.
\end{lemma}
\begin{proof}
 As the intersection
 $C_2\cap C_3\cap L_{\infty}=\{O^*\}$ and $I(C_2,C_3;O^*)=3$,
 the sum of the intersection numbers of $C_2\cap C_3$ is $3$
 in the affine space $\Bbb{P}^2\setminus L_{\infty}$. 
 The possible configurations of $\sing(D)$ are
 $\{3a_2\}$, $\{a_5+a_2\}$ and $\{a_8\}$.
 The singularity $a_8$ is locally irreducible 
 but the Milnor number of an irreducible quartics
 is less then or equal to $6$.
 Hence the configuration $\{a_8\}$ does not occur.
 \end{proof}

\begin{remark}{\rm{
 If $D$ is bi-tangent to $L_{\infty}$, then
 the configuration $\{a_5+a_2\}$ does not exist.
 Indeed,
 if the configuration $\{a_5+a_2\}$ exists,
 then $D$ can not be irreducible.
 If $D$ is a union of a line and a cubic,
 then a cubic can not have a bi-tangent line.
 If $D$ is a union of two conics which are tangent to $L_{\infty}$,
 then $D$ can not have any $a_2$ singularity.
 }}\end{remark}
 
Now we consider the case $c_1=\ell_2(0,1)=0$.
Then putting $c_2=\ell_3(0,1)$,
their defining polynomials are given as
\[
\begin{split}
  f_2(x,z)&=a_{00}\,z^2+\ell_2(1,0)\,xz+x^2,\\
  f_3(x,z)&=c_2\,z^2+x^3
 +\left(c_3\,x^2z
 +c_4\,xz^2+b_{00}z^3\right),\ \ c_3,c_4\in \Bbb{C}.
\end{split}
\]
Thus $C_2$ consists of two lines $\ell_1$ and $\ell_2$ such that
 $\ell_1\cap \ell_2=\{O^*\}$ and
 $C_3$ has an $a_2$ singularity at $O^*$ and $I(C_2,C_3;O^*)=4$.
 Then after an easy calculation, we have 
\[
 g(x,z)=-c_2^2\,z^2-2\,c_2\,x^3
\]
Hence $D$ has an $a_2$ singularity at $O^*$.
Thus we have:

 \begin{lemma}\label{l5}
  If $c_1=0$, 
  then $D$ has an $a_2$ singularity on $L_{\infty}$
  and $\sing(D)=\{2a_2+a_2^{\infty}\}$ or $\{a_5+a_2^{\infty}\}$.
   \end{lemma}
  \begin{proof}
  As $C_2\cap C_3\cap L_{\infty}=\{O^*\}$ and $I(C_2,C_3;O^*)=4$,
  the intersection  $C_2\cap C_3$
  generically consists of two points in the affine space
  $\Bbb{P}^2\setminus L_{\infty}$.
  By a similar argument of Lemma  $\ref{l4}$, 
  we have the assertion.
     \end{proof}

 \section{Statement of the Theorem.}\label{s3}
Let $Q$ be a quartic in $QL$-configurations.
For a quartic $Q$ in one of the $(1),\dots, (13)$ of Table $1$,
$Q$ is irreducible and
$Q$ is not irreducible for $(14),\dots, (19)$ of Table $1$.
Now our main results are the followings:

\begin{theorem}
 Let $Q$ be an irreducible quartic in one of the
 $QL$-configurations.
 \begin{enumerate}
\item[{\rm{(1)}}]
             For $Q$ in the case $(13)$,
             there exists no $(2,3)$ torus decomposition.

 \item[{\rm{(2)}}]
  For $Q$ in the case $(5)$,
   there exist five torus decompositions of type $(2,3)$
  whose three decompositions are  visible  decompositions  
       and two are invisible  decompositions.
\item[{\rm{(3)}}]
 For $Q$ in the remaining cases,             
there exists a unique $(2,3)$ torus decomposition
             for each case.
\end{enumerate}
\end{theorem}

Note that a quartic $Q$ for the case  $(5)$ is a $3$ cuspidal quartic.

\begin{remark}
 {\rm{
In section 9, we will show that
if a plane curve  has a $(2,3)$ torus decomposition,
then there exist infinite $(2,3)$ quasi torus decompositions.
 }}
\end{remark}

\begin{theorem}
For each quartic $Q$ of $(1),\dots, (12)$,
 there exists a line degeneration family of sextic
  $C(s):H_3(X,Y,Z,s)^2+H_2(X,Y,Z,s)^3=0$ which are
 $(2,3)$ torus curves such that $C(0)=Q+2L_{\infty}$.
 In particular,
 we have the divisibility $\Delta_{C(s)}(t)\mid \Delta_{Q}(t)$.
\end{theorem}

The divisibility $(\star)$ in Introduction also follows
 from Theorem 2 and Corollary 15 of \cite{okatan}.

\begin{proposition}\label{p2}
 For non-irreducible quartics $(14),\dots, (19)$ in Table $1$,
 we have the following:  
 \begin{enumerate}
\item[{\rm{(a)}}] There exist unique $(2,3)$ torus decompositions for
           quartics $(14)$ and $(15)$ and
           their decompositions are represented as
           visible decompositions.
\item[{\rm{(b)}}] The quartic $(16)$ does not admit any torus decompositions.
\item[{\rm{(c)}}] There exist unique $(2,4)$ torus decompositions for
           the quartics $(17)$, $(18)$ and $(19)$.
           Their decompositions are represented as
           invisible  decompositions. 
 \end{enumerate}
 \end{proposition}

\begin{remark}{\rm{
For the quartics $(13)$ and $(16)$,
 there are not torus decomposition.
By the classifications of singularities in \S 2,
 their singularities
do not occur as the quartics with
 visible or invisible $(2,3)$ torus decompositions.
}}\end{remark}

\section{The proof of Theorem 1 of $(3)$}\label{s4}
\subsection{Strategy}
 There are $13$ configurations  of singularities of
the quintic $Q+ L_{\infty}$ as in $(1)$, $\dots$, $(13)$ in Table 1.
  We divide these quintic into $5$ cases ${\rm{(i)}},\dots, {\rm{(v)}}$
  as in Introduction.
Note that the case ${\rm{(iii)}}$ does not appear when $Q$ is irreducible.

By the classification of the singularities for
the visible and invisible factorizations in \S 2,
for the quartics $(1),\dots, (12)$ except the case $(5)$,
the possible torus decomposition must be visible and
unique if it exists.
The quartic $(5)$ has an exceptional property.
It has both visible and invisible torus decompositions.
Thus we treat this case in the next section.

First, we construct explicit quartics $Q:=\{F=0\}$ with 
the prescribed properties at infinity.
By the action of ${\rm{PSL}}(3,\Bbb{C})$ of $\Bbb{P}^2$,
   we can put the singularities at fixed locations.
   Then we construct the respective
torus decompositions in \S 2.\\

\noindent
{\bf{Step 1. Construction of an explicit quartic $Q$.}}
By the classification of the singularities for  invisible  decomposition
case
(Lemma 4 and Lemma 5), the quartics in cases (1)$\sim$ (12) except the
case (5) can not have invisible torus decomposition.
 So we only need the possible visible decomposition for these quartics.
As the computations are boring and easy,
we explain the quartic $(1)$ in Table 1 in detail and
for the other cases we simply give the result of the computations.

 \noindent
{\bf{The quartic $(1)$ in Table 1.}}
 In this case, $L_{\infty}$ is a bi-tangent line of $Q$
and the singularity is $\sing(Q)=\{2a_2\}$.
We construct a quartic $Q$ with $2a_2$ which
$L_{\infty}$ is a bi-tangent line.
Let  $\Sigma(Q):=\{P_1,P_2\}$ be the singular locus of $Q$  and
let $Q\cap L_{\infty}:=\{R_1,R_2\}$ be the bi-tangent points.
 By the action of $\rm{PSL}(3,\Bbb{C})$ on $\Bbb{P}^2$,
we can put the locations of points:
\[
 P_1=[1:0:1],\quad
 P_2=[-1:0:1],\quad R_1=[1:1:0]
\]
and we may assume that the tangent directions at $P_1$ and $P_2$
are given as 
\[
\{x-1=0\},\quad \{x+1=0\} {\text{ \ respectively. }}
\]
We start from the generic quartic $F(X,Y,Z)=\sum_{\nu}c_\nu
X^{\nu_1}Y^{\nu_2}Z^{\nu_3}$ with $\nu=(\nu_1,\nu_2,\nu_3)$
with $\nu_1+\nu_2+\nu_3=4$.
The 
necessary conditions are
\[
 \begin{split}
&F(P_1)=\frac{\partial F}{\partial X}(P_1)
       =\frac{\partial F}{\partial Y}(P_1)
       =\frac{\partial F}{\partial YY}(P_1)
       =\frac{\partial F}{\partial XY}(P_1)
       =0,\\
&F(P_2)=\frac{\partial F}{\partial X}(P_2)
       =\frac{\partial F}{\partial Y}(P_2)
       =\frac{\partial F}{\partial YY}(P_2)
       =\frac{\partial F}{\partial XY}(P_2)
       =0,\\
&F(R_1)=\frac{\partial F}{\partial X}(R_1)=0.
 \end{split}
\]
Under the above conditions,
  we have $F(X,Y,0)=(X-Y)^2(X-\alpha\,Y)(X-\beta\, Y)$.
As $L_{\infty}$ is bi-tangent to $Q$,
we must have $\alpha=\beta$.
Thus we have $13$ equations of the coefficients of $F$.
By solving these equations,  $F$  has the following form:
\[
Q:\quad  F(X,Y,Z)=Z^4+2(Y^2-X^2)Z^2+t\,Y^3Z+(Y^2-X^2)^2=0,\quad t\ne 0.
\]
Then another bi-tangent point $R_2$ is $[-1:1:0]$.\\

\noindent
{\bf{Step 2. Torus decompositions.}}
Now we consider the possibilities of visible torus decompositions of $Q$.
Thus we assume that $F$ is written as follows: 
\[
   F(X,Y,Z)=F_2'(X,Y,Z)^2+F_1'(X,Y,Z)^3Z.
 \]
Two $a_2$ singularities are inner singularities of $Q$.
Hence we assume that $C_2\cap L=\{P_1,P_2\}$ and
$C_2$ is smooth at $P_1$ and $P_2$ where $C_2=\{F_2'=0\}$ and
$L=\{F_1'=0\}$.
Then we have
\[\begin{split}
 & F_1'(P_1)=F_1'(P_2)=0,\qquad F_2'(P_1)=F_2'(P_2)=0.
\end{split}
\]
Then $L$ is the line pass through $P_1$ and $P_2$.
Hence we get $F_1'(X,Y,Z)=s_1\,Y$ where $s_1\in {\Bbb C^*}$.
As $C_2$ is smooth at $P_i$,
the tangent directions of $C_2$ at $P_i$ must coincide with
that of $Q$ for $i=1,2$.
 Hence $F_2'$ also satisfies the following conditions:
\[
   \dfrac{\partial F_2'}{\partial Y}(P_1)=
             \dfrac{\partial F_2'}{\partial Y}(P_2)=0.
\]
Then $\sing(Q)=\{2a_2\}$ by Lemma \ref{l1}.
 As $R_1\in Q$, $C_2$ passes through $R_1$.
 Namely $F_2'$ satisfies the condition $F_2'(R_1)=0$.
Then $F_2'$ takes the following form: 
\[
 F_2'(X,Y,Z)=s_2\,(X^2-Y^2-Z^2),\ \ s_2\in \Bbb{C}^*.
\]
Note that $C_2$ also passes through
another bi-tangent point $R_2$.
Hence $Q$ satisfies the condition (i) by Lemma \ref{l1}.
Therefore we get the family of quartics with 
 visible factorizations:
\[
 F(X,Y,Z)=s_2^2\,(X^2-Y^2-Z^2)^2+s_1^3\, Y^3\,Z=0.
\]
Finally, we put $s_1^3=t$ and $s_2^2=1$.
Then we can see easily 
\[
\begin{split}
  (X^2-Y^2-Z^2)^2+t\,Y^3Z
 &=Z^4-2 \left(X^2+Y^2 \right)Z^2+t\,Y^3Z+ (X^2+Y^2)^2\\
      &=F(X,Y,Z). 
\end{split}
\]

\noindent
{\bf{Step 3. Uniqueness.}}
By the classification of the singularities for
the visible and invisible factorizations in \S 2,
we can see easily that the possible torus decompositions are visible.
Then two $a_2$ singularities
must be inner singularities and the  
corresponding curves are uniquely determined by the above arguments.

\begin{figure}[H]
 \begin{center}
  \includegraphics[scale=0.75]{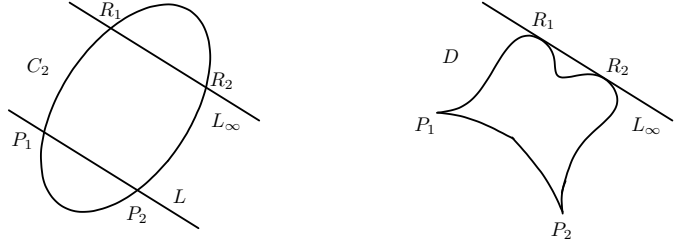}
 \end{center}
 \caption{The quartic (1) in Table 1} \label{s1}
\end{figure}

For the other quartics, we only give the results of calculations.
%
 {\renewcommand{\arraystretch}{1.36}
\begin{center}
\begin{longtable}{|c|c|}
 \hline
\multicolumn{2}{|c|}{Quartic (1)}\\ \hline
 Singularities & $2a_2$ at $[1:0:1]$, $[-1:0:1]$ \\  
 \hline
$Q\cap L_\infty$ &  bi-tangent at $[1:1:0]$, $[-1:1:0]$  \\  \hline
Torus decomposition & $(X^2-Y^2-Z^2)^2+t\,Y^3Z=0$\\ \hline

\multicolumn{2}{|c|}{Quartic (2)}\\ \hline
 Singularities &  $2a_2$ at $[1:0:1]$, $[-1:0:1]$\\
 \hline
$Q\cap L_\infty$ & tangent multiplicity 4  at $[0:1:0]$.
 \\  \hline
Torus decomposition & $(X^2-Z^2)^2+t\,Y^3Z=0$.   \\ \hline

\multicolumn{2}{|c|}{Quartic (3)}\\ \hline
 Singularities &  $2a_2+a_1$ at
      $[0:0:1]$, $[1:1:1]$, $[-1:1:0]$.
     \\  \hline
$Q\cap L_\infty$ & bi-tangent at
      $[\sqrt{3}:1:0]$, $[-\sqrt{3}:1:0]$.
 \\  \hline
Torus decomposition & $(4Z^2-6YZ-X^2+3Y^2)^2+16\,(Y-Z)Z=0$.
     
     \\ \hline

\multicolumn{2}{|c|}{Quartic (4)}\\ \hline
 Singularities & $2a_2+a_1$ at
      $[0:0:1]$, $[1:1:1]$, $[-1:1:0]$.
     \\  \hline
 $Q\cap L_\infty$ &
     tangent multiplicity 4  at $[0:1:0]$ 
 \\  \hline
Torus decomposition & $(2Z^2+X^2-3YZ)^2+4\,(Y-Z)Z=0$.
     
     \\ \hline
 
\multicolumn{2}{|c|}{Quartic (6)}\\ \hline
 Singularities &
$a_2+a_3^{\infty}$ at
      $[0:0:1]$, $[-1:1:0]$.
     \\  \hline
 $Q\cap L_\infty$ &
  singular at $[-1:1:0]$ and tangent at $[1:1:0]$.
      \\  \hline
Torus decomposition & $(X^2-2\,XZ-Y^2)^2+t\,(X+Y)Z=0$.
     
     \\ \hline

\multicolumn{2}{|c|}{Quartic (7)}\\ \hline
 Singularities &
     $a_5$ at $[0:0:1]$.
     \\  \hline
 $Q\cap L_\infty$ &
   bi-tangent at $[1:1:0]$, $[-1:1:0]$.
 \\  \hline
Torus decomposition & $(XZ+s(X^2-Y^2))^2+s(ts-2)\,X^3Z=0$.
     
     \\ \hline
 
\multicolumn{2}{|c|}{Quartic (8)}\\ \hline 
Singularities &  $a_5$ at $[0:0:1]$.
     \\ \hline
$Q\cap L_\infty$ &  tangent multiplicity 4  at $[1:0:0]$.
      \\  \hline
Torus decomposition & $(XZ-s\,Y^2)^2+t\,X^3Z=0$\\ \hline

\multicolumn{2}{|c|}{Quartic (9)}\\ \hline
 Singularities & $a_6^{\infty}$ at $[0:1:0]$.
     \\ \hline
$Q\cap L_\infty$ &  singular at $[0:1:0]$ and tangent at $[-1:1:0]$.
           \\  \hline
Torus decomposition &
     $(t_2^2Z^2+t_3XZ-t_2X(X+Y))^2$\hspace{3cm} \\ \ &
     \hspace{3.5cm}$+t_2(2t_3+t_1t_2)\,X^3Z=0$.
     \\ \hline

\multicolumn{2}{|c|}{Quartic (10)}\\ \hline
 Singularities & $a_4^{\infty},a_2$ at  $[0:0:1]$, $[0:1:0]$.
          \\ \hline
$Q\cap L_\infty$ &   singular at $[0:1:0]$.
           \\  \hline
Torus decomposition & $(YZ-t_2X^2)^2+t\,X^3Z=0$
     \\ \hline

\multicolumn{2}{|c|}{Quartic (11)}\\ \hline
 Singularities &  $e_6$ at $[0:0:1]$.
               \\ \hline
$Q\cap L_\infty$ &   bi-tangent at $[1:1:0]$, $[-1:1:0]$.
           \\  \hline
Torus decomposition & $(X^2+Y^2)^2+t\,X^3Z=0$
     \\ \hline

\multicolumn{2}{|c|}{Quartic (12)}\\ \hline
 Singularities &  $e_6$ at $[0:0:1]$.
               \\ \hline
$Q\cap L_\infty$ & tangent multiplicity 4  at $[1:0:0]$.
           \\  \hline
Torus decomposition &  $Y^4+t\,X^3Z=0$.
     \\ \hline

\end{longtable}
\end{center}
}

\section{Proof of Theorem  1 of $(2)$.}\label{s5}
In this section, we consider the exceptional case (5) in Table $1$.
This case (5) is exceptional as the classification of the singularities tell us
that it may have a invisible case 
as well as visible decompositions.
 Let $Q=\{F=0\}$ be a quartic with $\sing(Q)=\{3a_2\}$.
Then, by the class formula (\cite{Okadual}),
$Q$ has a unique  bi-tangent line and
we take $L_{\infty}$ as the bi-tangent line of $Q$.
We put the singular locus $\Sigma(Q)=\{P_1,P_2,P_3\}$
and the intersection  $Q \cap L_{\infty}:=\{R_1, R_2\}$.
By the action of $\rm{PSL}(3,\Bbb{C})$ on $\Bbb{P}^2$,
we can put the locations:
 \[
  P_1=[1:0:1],\quad
  P_2=[-\frac{1}{2}:\frac{\sqrt{3}}{2}:1],\quad
  P_3=[-\frac{1}{2}:-\frac{\sqrt{3}}{2}:1],\quad 
  R_1=[I,1,0]
   \]
   where $I=\sqrt{-1}$.
By direct computations,
 the defining polynomial $F$ of $Q$ is obtained by
\[
 F(X,Y,Z)=Z^4-6(X^2+Y^2)Z^2+8(X^2-3Y^2)XZ-3(X^2+Y^2)^2.
\]
Then another bi-tangent point $R_2$ is $[-I:1:0]$.
Note that there is 
 no free parameters left.

Now we consider two transformations
$\sigma$, $\tau:\Bbb{P}^2\to \Bbb{P}^2$ which are defined as 
\[
 \sigma(X:Y:Z):=(X:-Y:Z),\quad
   \tau(X:Y:Z):=(X:Y:Z)A,
\]
where
\[
 A= \begin{pmatrix}
         \cos \theta & -\sin  \theta & 0 \\
         \sin \theta & \cos   \theta & 0\\
          0&0&1              
     \end{pmatrix}, \quad \theta=-\frac{2}{3}\pi. 
\]
Consider the subgroup $G$ of ${\rm{PSL(3;\Bbb{C})}}$
generated by $\sigma$ and $\tau$.
Observe that $G\cong S_3$ and $\sigma^2=\tau^3=(\sigma\tau)^2=e$.
Then we can see that $L_{\infty}$ and $Q=\{F=0\}$ are
stable under the action of $G$:
\[
 F(X,Y,Z)=F(\sigma(X,Y,Z))=F(\tau(X,Y,Z)).
\]
We observe also the following.
\[
\begin{split}
& \sigma(R_1)=R_2,\  \sigma(R_2)=R_1,\
 \sigma(P_1)=P_1, \ \sigma(P_2)=P_3, \  \sigma(P_3)=P_2.\\
& \tau(R_i)=R_i,\ i=1,2,\qquad 
 \tau(P_i)=
 \begin{cases}
  P_{i+1} & {\text{if $i=1,2$}}\\
  P_1     & {\text{if $i=3$.}}
  \end{cases}
\end{split}
\]

\noindent
{\bf{Visible factorization.}}
Now we consider the possibilities of $(2,3)$ visible factorization
of $Q=\{F=0\}$.
We assume that $F$ is written as follows: 
\[
   F(X,Y,Z)=F_2'(X,Y,Z)^2+F_1'(X,Y,Z)^3Z.
 \]
In this case,
two of $P_1$, $P_2$, $P_3$ must be inner singularities
and the rest is an outer singularity.
Thus we have three possible cases for these choices:
\begin{enumerate}
 \item[{\rm{(1)}}] $P_1$ is an outer singularity and
       $P_2$, $P_3$ are inner singularities.
 \item[{\rm{(2)}}] $P_2$ is an outer singularity and
       $P_1$, $P_3$ are inner singularities.
 \item[{\rm{(3)}}] $P_3$ is an outer singularity and
       $P_1$, $P_2$ are inner singularities.
 \end{enumerate}

First we assume  the case $(1)$.
Then  $L=\{F_1'=0\}$ and $C_2=\{F_2'=0\}$ are satisfy the following.
\begin{itemize}
 \item $P_1$ is an outer singularity.
 \item $L$ is the line passing $P_2$ and $P_3$.
 \item $C_2$ passes through $P_2$, $P_3$, $R_1$ and $R_2$. 
 \end{itemize}
Then the defining polynomials $F_1'$ and $F_2'$ are obtained by
\[
 F_1'(X,Y,Z)=-\frac{1}{3}t^2(Z+2X),\quad 
 F_2'(X,Y,Z)=\frac{t^3}{6}(Z^2+4XZ+Y^2+X^2).
\]
We take $t$ as one of the solutions  $t^6+108=0$.
Then $F$ is decomposed into 
\[
\tag{V-1}
  F(X,Y,Z)=-3(Z^2+4XZ+Y^2+X^2)^2+4(Z+2X)^3Z.
\]
Note that $C_2$, $L$, $P_1$ and $\{P_2, P_3\}$
 are stable by the action of $\sigma$.

Next we consider the case $(2)$.
The singular locus $\Sigma(Q)$ is stable
by the action of $\tau$ and $\tau(C_2\cap L)=\{P_3,P_1\}$.
Hence $P_2$ is the outer singularity. 
Thus we have
 \[
 \tag{V-2}
    \begin{split}
 F(X,Y,Z)&= F(\tau(X,Y,Z))=F_2(\tau(X,Y,Z))^2+F_1(\tau(X,Y,Z))^3Z\\
& =-3(Z^2-2XZ-2\sqrt{3}YZ+X^2+Y^2)^2+4(Z-X-\sqrt{3}Y)^3.
\end{split}
 \]
By the same argument, we have one more different torus decomposition:
  \[
    \tag{V-3}
      \begin{split}
 F(X,Y,Z)&=F(\tau^2(X,Y,Z))=F_2'(\tau^2(X,Y,Z))^2+F_1'(\tau^2(X,Y,Z))^3Z\\
      & =-3(Z^2-2XZ+2\sqrt{3}YZ+X^2+Y^2)^2+4(Z-X+\sqrt{3}Y)^3.
       \end{split}
  \]
Thus we have three different torus decompositions
(V-1), (V-2) and (V-3).\\

\noindent
{\bf{Invisible factorization.}}
Next we consider $(2,3)$ invisible factorization (\S 2): 
 \[
 Z^2F(X,Y,Z)=F_2(X,Y,Z)^3-F_3(X,Y,Z)^2
\]
where $F_2$ and $F_3$ are defined by
\begin{eqnarray*}(\sharp)\quad   
\begin{split}
 F_2(X,Y,Z)&=\ell_1(X,Y)^2+\ell_2(X,Y)\,Z+a_{00}\,Z^2,\\ 
 F_3(X,Y,Z)&=\ell_1(X,Y)^3+\frac{3}{2}\,\ell_1(X,Y)\,\ell_2(X,Y)\,Z
                 +\ell_3(X,Y)\,Z^2+b_{00}\,Z^3
 \end{split}
\end{eqnarray*}
where $\ell_i$ is a linear form for $i=1,2,3$.
By the argument in \S 2.2,
the singularity locus $P_1$, $P_2$ and $P_3$ are
  inner singularities.
Hence we have the conditions:
\[
 (*_1)\quad F_2(P_i)=F_3(P_i)=0,\quad i=1,2,3.
\]
Moreover one of the bi-tangent points is obtained by
the intersection point $\{\ell_1=0\}\cap L_{\infty}$.

First we assume that $\{\ell_1=0\}\cap L_{\infty}=\{R_1\}$.
 By solving conditions
 $(*_1)$ and  $\ell_1(R_1)=0$,
 we have $a_{00}=0$, $b_{00}=t^3/2$ and  
\[
\ell_1(X,Y)=t\,(X-IY),\quad
\ell_2(X,Y)=-t^2\,(X+IY),\quad
\ell_3(X,Y)=0.
\]
Thus $F_2$ and $F_3$ are given by
\[\begin{split}
 F_2(X,Y,Z)&=t^2(X-IY)^2-t^2(X+IY)Z,\\
 F_3(X,Y,Z)&=\frac{t^3}{2}\left( Z^3-3(X^2+Y^2)Z+2(X-IY)^3 \right)
      \end{split}
\]
 Note that $C_2=\{F_2=0\}$ and
           $C_3=\{F_3=0\}$ are stable by the action $\sigma$.
Then we have 
\[
 \frac{t^6}{4}Z^2F(X,Y,Z)=F_2(X,Y,Z)^3-F_3(X,Y,Z)^2.
\]
Hence taking $t$ as one of the solutions  $t^6=4$.
Thus we have an invisible torus decomposition: 
\[\tag{In-1}
\begin{split}
Z^2F(X,Y,Z)&=((X+IY)Z-(X-IY)^2)^3+(Z^3-3(X^2+Y^2)Z+2(X-IY)^3)^2.
\end{split}
\]

Next we consider the case $\{\ell_1=0\}\cap L_{\infty}=\{R_2\}$.
As the  singular locus $\Sigma(Q)$ is stable
by the action of $\sigma$ and
$\sigma(R_1)=R_2$,
we have another invisible torus decomposition from (In-1):
%
           \[\tag{In-2}
 Z^2F(\sigma(X,Y,Z))=((X-IY)Z-(X+IY)^2)^3+(Z^3-3(X^2+Y^2)Z+2(X+IY)^3)^2.
\]
Thus we have two different invisible torus decompositions
(In-1) and (In-2).

\begin{figure}[H]
 \begin{center}
  \includegraphics[scale=0.65]{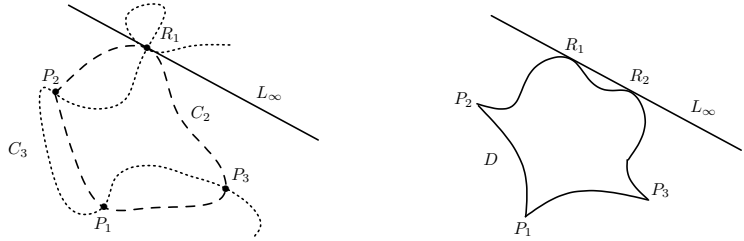}
 \end{center}
 \caption{Invisible factorization (In-1) of the quartic  (5)}
 \label{a3}
\end{figure}
\vspace{-0.4cm}

We have shown in the above argument
 that the three visible decompositions move each other 
by the action of $\sigma$:
\[\begin{matrix}
   {\rm{(V{\text{-}}1)}}&\overset{\sigma}{\to} &
   {\rm{(V{\text{-}}2)}}&\overset{\sigma}{\to} &
   {\rm{(V{\text{-}}3)}}.
   \end{matrix}
   \]
  We also showed that
   the two invisible decompositions move each other 
by the action of $\tau$:
\[
 \begin{matrix}
   &  {\rm{(In{\text{-}}1)}}&\overset{\tau}{\to} 
               &{\rm{(In{\text{-}}2)}}.
   \end{matrix}
\]

In this case, there exist other quasi torus decompositions.
 We discuss in section 9.



\section{Proof of Theorem 2}\label{s6}
Let $U$ be an open neighborhood of $0$
and let $C(t)=\{F_{2,3}(X,Y,Z,t)=0\}$ $(t\in U)$ 
be a $(2,3)$ torus curve which is defined by
 \[
  F_{2,3}(X,Y,Z,t)=H_3(X,Y,Z,t)^2+H_2(X,Y,Z,t)^3.
 \]
 Fix $B=[0:1:0]$ in $\Bbb{P}^2$.
 We consider the family $\{C(t)\}_{t\in U}$ which has
 $B$ as the base point with multiplicity $2$. 
We assume that
the defining polynomials have the following form: 
\[\tag{$*_2$}
\begin{cases}
\begin{split}
 H_2(X,Y,Z,t)&=F_1'(X,Y,Z)Z+t\,X K_1(X,Y),\\
 H_3(X,Y,Z,t)&=F_2'(X,Y,Z)Z+t\,X K_2(X,Y)
\end{split}
\end{cases}
\]
where $F_i'(X,Y,Z)$ and $K_i(X,Y)$ homogeneous polynomials of degree $i$ 
such that $K_i(0,1)\ne 0$ for $i=1,2$.
Then $B$ is in $C(t)$ with multiplicity $2$ for $t\ne 0$ and
 $\{C(t)\}_{t\in U}$ degenerates into $C(0)=D+2L_{\infty}$
where $D=\{G=0\}$ is a visible factorization of torus curve
which is defined as
\[
   G(X,Y,Z)=F_2'(X,Y,Z)^2+F_1'(X,Y,Z)^3Z.
 \]
Thus we can construct
a line degeneration of $(2,3)$ torus sextic.
For example, we give line degeneration for the case $(1)$.
Let $Q=\{F=0\}$ be the quartic which is satisfies the condition of $(1)$.
In \S 4, we obtained the defining polynomials $F$, $F_1$ and $F_2$ as
 \[\begin{split}
Q:&\quad  F(X,Y,Z)=Z^4+2(Y^2-X^2)Z^2+t\,Y^3Z+(Y^2-X^2)^2=0, \\    
L:&\quad F_1'(X,Y,Z)=s_1\,Y=0,\\
C_2:&\quad  F_2'(X,Y,Z)=s_2\,(X^2+Y^2-Z^2)=0   
   \end{split}
 \]
where $s_1$ and $s_2$ are  one of the solution
$s_1^3=t$ and $s_2^2=1$ respectively.
Let $K_i(X,Y)$ be any  homogeneous polynomial of degree $i$
such that $K_i(0,1)\ne 0$ for $i=1$, $2$.
We take $H_2$ and $H_3$ as $(*_2)$ and then
the family $\{C(t)\}_{t\in U}$  degenerates into $C(0)=Q+2L_{\infty}$.

\section{Degenerate families of $QL$-configurations.}\label{s7}
Let $\mathcal{QL}$ be the set of quaritcs of $QL$-configurations
and let $L_{\infty}$ be the fixed line at infinity.
We consider the subset $\mathcal{QL}(n)$ of
$\mathcal{QL}$ which is the set of quartic $Q$ whose
configuration $Q\cup L_{\infty}$ is of the type $(n)$ in Table 1
for $n=1,\dots, 12$.
It is sesy to see that
$\mathcal{QL}(n)$ is a connected subspace of the space of quartic
under the canonical topology.
This implies that for any $Q$, $Q'\in \mathcal{QL}(n)$
the topology of
$\Bbb{C}^2\setminus Q$ and $\Bbb{C}^2\setminus Q'$
are homeomorphic.

For the comparison of the topology of $\Bbb{C}^2\setminus Q$
and $\Bbb{C}^2\setminus Q'$ for $Q\in \mathcal{QL}(n)$ and
$Q'\in \mathcal{QL}(m)$,
we consider the degeneration problem among these subsets.
Suppose that there exists an analytic family $Q(s)$ $(s\in U)$ of
quartic for an open set $U$ of the origin $0\in  \Bbb{C}$
such that $Q(s)\in \mathcal{QL}(n)$ for $s\in U\setminus \{0\}$
and  $Q(0)\in \mathcal{QL}(m)$ for some $n$, $m$ with $n\ne m$. 
In particular, $Q(s)\cup L_{\infty}\to Q(0)\cup L_{\infty}$.
Then, by an degenerate properties (\cite{OkaSurvey}),
we have the surjectivity
 $\pi_1(\Bbb{C}\setminus Q(0))\twoheadrightarrow
 \pi_1(\Bbb{C}\setminus Q(s))$ $(s\ne 0)$
and the divisibility of the Alexander polynomials
 $\Delta_{Q(s)}(t) \mid \Delta_{Q(0)}(t)$.
 We denote this situation as $\mathcal{QL}(n) \to \mathcal{QL}(m)$.

\begin{proposition}
 There exist the following degeneration families among the set
 $\mathcal{QL}(n)$, $n=1,\dots, 12$.
 \begin{figure}[H]
  \[
 \begin{matrix}
  && && &\mathcal{QL} (10)&&&&\\
 & &&&&\uparrow\\
  & && \mathcal{QL} (4)&\leftarrow &\mathcal{QL} (2) &\to &\mathcal{QL} (8)& \to &\mathcal{QL} (12)\\
&&& \uparrow&&  \uparrow && \uparrow && \uparrow\\
  &  \mathcal{QL} (5)&\leftarrow &\mathcal{QL} (3)&\leftarrow &\mathcal{QL} (1) &\to &\mathcal{QL} (7)& \to &\mathcal{QL} (11)\\
&&&&&\downarrow \\
  &  &&&& \mathcal{QL} (6)& \to &\mathcal{QL} (9)
 \end{matrix}
\]
 \end{figure}
 \noindent

 This follows from the following explicit degenerations:
 \begin{enumerate}
  \item[{\rm{(1)}}]
        There exist a family of quintic
       $\{C_{s,t,u}\} \subset \mathcal{QL}(1)$ where
               $C_{s,t,u}=Q_{s,t,u}+L_{\infty}$
        $(stu\ne 0)$ with
        $3$ parameters such that
        \[\begin{split}
        & C_{0,t,u}\in \mathcal{QL}(7),\quad
         C_{0,t,0}\in \mathcal{QL}(11)
         {\text{\ \  for \ $t\ne  0$,}}\\  
        & C_{s,0,u}\in \mathcal{QL}(2), \quad
         C_{0,0,u}\in \mathcal{QL}(8), \quad
         C_{0,0,0}\in \mathcal{QL}(12).
         \end{split}
        \]
  \item[{\rm{(2)}}]
       There exist a family of quintic
       $\{C_{s,t,u}\} \subset \mathcal{QL}(1)$ where
               $C_{s,t,u}=Q_{s,t,u}+L_{\infty}$
        $(stu\ne 0)$ with
        $3$ parameters such that
     \[\begin{split}
        & C_{0,t,u}\in \mathcal{QL}(3),\quad
         C_{0,t,0}\in \mathcal{QL}(5)
         {\text{\ \ for \ \ $t\ne  0$,}}\\  
        & C_{s,0,u}\in \mathcal{QL}(2), \quad
        C_{0,0,u}\in \mathcal{QL}(4){\text{\  \ for  \ $u\ne  0$}}.
          \end{split}
        \]
  \item[{\rm{(3)}}]
           There exist a family of quintic
       $\{C_{s,t}\} \subset \mathcal{QL}(1)$ where
               $C_{s,t}=Q_{s,t}+L_{\infty}$ 
        $(st\ne 0)$ with
        $2$ parameters such that
\[
 C_{0,t}\in \mathcal{QL}(6),\quad
         C_{0,0}\in \mathcal{QL}(9).
\]
        
  \item[{\rm{(4)}}]
                     There exist a family of quintic
       $\{C_{s,t}\} \subset \mathcal{QL}(1)$ where
               $C_{s,t}=Q_{s,t}+L_{\infty}$ 
        $(st\ne 0)$ with $2$ parameters such that
        \[
         C_{0,t}\in \mathcal{QL}(2),\quad
         C_{0,0}\in \mathcal{QL}(10).
        \]
  \end{enumerate}
\end{proposition}
 \noindent
For a proof,
 we give explicit defining equations of quartic for each case. 
\[\tag{1}
 \begin{split}
  F_{s,t,u}(X,Y,Z)&=s^4Z^4-2s^2uYZ^3
           +(2s^2(t^2Y^2-X^2)+u^2Y^2)Z^2\\
& \hspace{3cm}          +(Y^2+2u(X^2-t^2Y^2))YZ
           +(X^2-t^2Y^2)^2.
 \end{split}
\]
\[\tag{2}
 \begin{split}
  F_{s,t,u}(X,Y,Z)&=s(2+s)Z^4-3sXZ^3
           +\frac{1}{4}(x^2(3s+8u+8su)-y^2(s+1)(8u+3))Z^2\\
& \hspace{1cm}        -\frac{1}{8}(u(x^2-t^2y^2)+(x^2-9t^2y^2))XZ
  +\frac{1}{64}(8u+3)^2(X^2-t^2Y^2)^2.
 \end{split}
\]

\[\tag{3}
 \begin{split}
  F_{s,t}(X,Y,Z)&=(t^3+1)Z^4+3t^2\ell_1(X,Y,s)Z^3
  +(3\ell_1(X,Y,s)^2t-2(X^2-Y^2))Z^2\\
& \hspace{4cm}          +\ell_1(X,Y,s)^3Z
           +(X^2-Y^2)^2
 \end{split}
\]
where $\ell_1(X,Y,s)=X-sY-Y$.
\[
 \tag{4}
 \begin{split}
  F_{s,t}(X,Y,Z)&=(X+Y)Z^2+
  ((X-sY)^3-2(X+Y)(t^2Y^2-X^2))Z+(t^2Y^2-X^2)^2.
 \end{split}
\]

 \noindent
 {\bf{Example.}}
 Let $Q_1\in \mathcal{QL}(1)$ and $Q_5\in \mathcal{QL}(5)$
 be $QL$-configurations of types $(1)$ and $(5)$
 whose fundamental gruops and
 Alexander polynomials are given as follows (\cite{Y}):
 \[
  \begin{split}
   & \pi_1(\Bbb{C}^2\setminus Q_1)\cong
   \langle a, b \mid R_1 \rangle, \qquad 
    \pi_1(\Bbb{C}^2\setminus Q_5)\cong
   \langle a, b,c  \mid R_1,\,  R_2, \, R_3
                   \rangle\\
   &R_1: aba=bab,\
    R_2: bcb = cbc, \
    R_3: c(b^{-1}ab)c=(b^{-1}ab)c(b^{-1}ab)\\
&  \Delta_{Q_1}(t)=t^2-t+1,\qquad
   \Delta_{Q_5}(t)=(t^2-t+1)^2.
   \end{split}
 \]
By above Proposition, we have $\mathcal{QL}(1)\to \mathcal{QL}(5)$. 
Hence we have the surjectivity
$\pi_1(\Bbb{C}^2\setminus Q_5)\twoheadrightarrow
\pi_1(\Bbb{C}^2\setminus Q_1)$ and
the divisibility $\Delta_{Q_1}(t)\mid \Delta_{Q_5}(t)$.

\section{Non-irreducible $QL$-configurations}\label{s8}
In this section,
 we consider 
torus decompositions of non-irreducible \linebreak
$QL$-configurations
 for the quartics $(14),\dots,  (19)$.
Then we will get defining polynomials and torus decompositions
 by the same argument of irreducible case.
Recall that the situations of each case:
\vspace{0.1cm}
 \begin{center}
\begin{tabular}{|c|c|c|c|c|c|c|c|}
\hline
No. &\quad $\sing(Q)$\quad   \quad & $Q\cap L_{\infty}$  & irreducible components \\ \hline
(14) & $a_3^{\infty}+a_2+a_1$ & (ii)&a cuspidal cubic and a line  \\
(15) & $a_5+a_1$ & (i)    & two conics   \\
(16) & $a_5^{\infty}+a_2^{\infty}$ & (iii) & a cuspidal cubic and a line  \\ 
(17)  & $2a_3^{\infty}$ & (iii) &two conics   \\
(18) & $a_7^{\infty}$ & (v)  &two conics  \\
(19) & $2a_3^{\infty}+a_1$ & (iii)& a conic and two lines \\
                     \hline
\end{tabular}\\[0.2cm]
Table 3. \\[1mm]
\end{center}

\subsection{Invisible factorization of (2,4) torus curves}\label{s9}
Let $D$ be an invisible factorization of $(2,4)$ torus curve
which satisfies the following.
\[
\begin{split}
 F_{2,4}(X,Y,Z)&=F_2(X,Y,Z)^4-F_4(X,Y,Z)^2\\
               &=(F_2(X,Y,Z)^2-F_4(X,Y,Z))(F_2(X,Y,Z)^2+F_4(X,Y,Z))\\
               & =Z^4G(X,Y,Z).
\end{split}
\]
By the same argument in \S 2.2,
 we can assume that the forms of $F_2$ and $F_4$ are
 \[\begin{split}
  F_2(X,Y,Z)&=F_2^{(2)}(X,Y)+F_{2}^{(1)}(X,Y)Z+a_{00}Z^2,\quad
    \deg F_2^{(i)}=i,\\
  F_4(X,Y,Z)&=F_2(X,Y,Z)^2-c\,Z^4,\quad
  c=b_{00}-a_{00}^2\ne 0.   
   \end{split}
 \]
Then $D=\{G=0\}$ is defined by
 \[
  D:\quad G(X,Y,Z)=F_2(X,Y,Z)^2-c'Z^4=0,\quad c'\ne 0.
 \]
By the form of the defining polynomial of $D$,
the inner singularities of $D$ is on $L_{\infty}$.
Singularities of $D$ are described as follows.
 \begin{lemma}Under the above  notations,
  $D$ has the following singularities.
  \begin{enumerate}
   \item[{\rm{(1)}}] If $C_2$ is smooth at $P\in C_2\cap L_{\infty}$,
         then $(D,P)\sim a_{3\iota-1}$ where $\iota=I(C_2,L_{\infty};P)$. 
   \item[{\rm{(2)}}] If $C_2$ is singular at $P\in C_2\cap L_{\infty}$, then
         $D$  consists of four lines.
   \item[{\rm{(3)}}] If $P\in D$ is an outer singularity, then $(D,P)\sim a_1$. 
  \end{enumerate}
  \end{lemma}
  \noindent
  Our proof is done in the same way as
  Lemma \ref{l1} and \cite{Oka-Pho2}.

\subsection{Torus decompositions of non-irreducible $QL$-configurations}
In this section,
we show the possibilities of
torus decompositions for non-irreducible $QL$-configurations.
Our proof is similar to the cases of  irreducible $QL$-configurations.
For the quartics  $(14)$ and $(15)$,
 we use $(2,3)$ visible factorizations.
For the quartics $(17)$, $(18)$ and $(19)$,
 we use $(2,4)$ invisible factorizations.\\

 {\renewcommand{\arraystretch}{1.36}
\begin{center}
\begin{longtable}{|c|c|c|c|c|c|c|c|}\hline
\multicolumn{2}{|c|}{Quartic (14)}\\ \hline
Singularities & $a_3^{\infty}+a_2+a_1$ at
      $[1:1:0]$, $[0:0:1]$, $[-1:0:1]$.
     \\ \hline
$Q\cap L_\infty$ &  singular at $[1:1:0]$ and tangent at $[-1:1:0]$.
           \\  \hline
Torus decomposition &
      $(X^2-Y^2-XZ+3YZ)^2+4\,(X-Y)^3Z=0$.
     \\ \hline

\multicolumn{2}{|c|}{Quartic (15)}\\ \hline
 Singularities & $a_5+a_1$ at
       $[0:0:1]$, $[-1:0:1]$.
     \\ \hline
$Q\cap L_\infty$ &  bi-tangent at $[1:1:0]$ and $[-1:1:0]$.
           \\  \hline
Torus decomposition &
      $(X^2-Y^2+XZ)^2+4\,X^3Z=0$.
     \\ \hline

\multicolumn{2}{|c|}{Quartic (17)}\\ \hline
 Singularities &  $2a_3$ at
       $[1:1:0]$ and  $[-1:0:0]$.
     \\ \hline
$Q\cap L_\infty$ & singular at        $[1:1:0]$ and  $[-1:0:0]$.
                \\  \hline
Torus decomposition &
      $\frac{1}{64}(2X^2-2Y^2-t_2Z^2)^4$ \\
                       \ & $-
      (\frac{1}{8}(2X^2-2Y^2-t_2Z^2)^2+\frac{1}{4}(4t_1-t_2^2)Z^4)^2$.
         \\ \hline

 \multicolumn{2}{|c|}{Quartic (18)}\\ \hline
Singularities &  $a_7$ at       $[0:1:0]$.
     \\ \hline
$Q\cap L_\infty$ &  singular at $[0:1:0]$.
           \\  \hline
Torus decomposition &
{\begin{small}
 $\left(Z^2-\frac{2}{a_{01}c_2}(c_3X-c_2Y)Z-2\frac{c_2}{a_{01}}X^2\right)^4$\end{small}} \\
 \ &
{\begin{small}
  $-      \left(\left(Z^2-\frac{2}{a_{01}c_2}(c_3X-c_2Y)Z
      -2\frac{c_2}{a_{01}}X^2\right)^2
      +2\frac{4a_{00}-a_{01}^2}{a_{01}^2}Z^4\right)^2$.
\end{small}}
  \\ \hline

  \multicolumn{2}{|c|}{Quartic (19)}\\ \hline
Singularities &  $2a_3+a_1$ at
       $[1:1:0]$, $[-1:1:0]$, $[0:0:1]$.
     \\ \hline
$Q\cap L_\infty$ & singular at    $[1:1:0]$ and  $[-1:1:0]$.
                \\  \hline
Torus decomposition &
 $\frac{1}{t^2}\Big{(} (-X^2+Y^2-\frac{1}{2}tZ^2)^4$ \\ \ &
\qquad \qquad     $  -\left((-X^2+Y^2-\frac{1}{2}tZ^2)^2
      -\frac{1}{2}t^2Z^4\right)^2\Big{)}$.
 \\
     \hline
 \end{longtable}
\end{center}
} 
\vspace{-2.0cm}

 \section{Infiniteness of $(2,3)$ quasi torus decompositions}
In this section,
we consider the possibilities of $(2,3)$
quasi torus decompositions of a plane curve which admits a $(2,3)$ torus
decomposition.
We assert:
\begin{proposition}\label{p1}
Let $C=\{f=0\}\subset \Bbb{C}^2$ be
 a $(2,3)$ torus curve of any degree. 
 Then $C$ has infinitely many $(2,3)$ quasi torus decompositions.
\end{proposition}
 \begin{proof}
  Suppose that $f(x,y)$ can be written as
$f(x,y)=h_0(x,y)^2-g_0(x,y)^3$.
  We put inductively
\[
 \begin{split}
g_{i+1}(x,y)&=-\frac{4}{3}\,h_{i}(x,y)^2+g_{i}(x,y)^3,\\
h_{i+1}(x,y)&=\frac{\sqrt{-3}}{9}h_i(x,y)(-8\,h_i(x,y)^2+9\,g_i(x,y)^3)  
 \end{split}
\]
  for $i\ge 0$. 
Then we claim that they satisfy the following equality:
\[\tag{$*_3$}
\left(\prod_{k=0}^ig_k(x,y)\right)^6f(x,y)=h_{i+1}(x,y)^2-g_{i+1}(x,y)^3,\quad
i\ge 0.
\]
Indeed,  by a simple calculation, we have
  \[
   h_{i+1}(x,y)^2-g_{i+1}(x,y)^3=g_i(x,y)^6\left(h_{i}(x,y)^2-g_{i}(x,y)^3\right).
  \]
  The assertion follows immediately from this equality.
  \end{proof}

\begin{remark}{\rm{
Let $C(t)$ be a family of curves given by
\[
  C(t):t\,h_{i+1}(x,y)^2-g_{i+1}(x,y)^3=0, \quad t\in \Bbb{C}.
\]
We thank Professor J. I. Cogolludo for informing us
the generic fiber of this pencil is not irreducible.}}
  \end{remark}
Now we study the location of singularities  of a family
of $(2,3)$ quasi torus decompositions which has the form  $(*_3)$. \\
We put $r_i(x,y):=\prod_{k=0}^ig_k(x,y)$ and
       $\Sigma_i:=\{h_i=0\}\cap \{g_i=0\}$  for $i\ge 0$. 
Then, by the definitions, we have the followings:
\begin{enumerate}
 \item[{\rm{(1)}}] $\Sigma_0$ is the set of inner singularities of $\{f=0\}$.
 \item[{\rm{(2)}}] $\Sigma_{i}\subset \{r_i=0\}$ for all $i \ge 0$. 
 \item[{\rm{(3)}}] $\Sigma_0\subset  \Sigma_1 \subset \cdots \subset \Sigma_i
       \subset \cdots$.
  \end{enumerate}
In particular,
$\{r_i=0\}$ contains the inner singularities of $\{f=0\}$
 for all $i\ge 0$.

By Proposition \ref{p1} and above observations,
it is important to study the existence of
$(2,3)$ torus decompositions which is obtained by
visible or invisible degenerations.
We are also interested in quasi torus decompositions
which does not come from a torus decomposition as in $(*_3)$.

We will give such an example $(2,3)$ quasi torus decomposition.
Let $Q=\{f=0\}$ be the three cuspidal quartic which has three
$(2,3)$ torus decompositions (V-1), (V-2) and (V-3) as in \S 5.
Recall that the $(2,3)$ torus decomposition (V-1) and
locations of singularities:
 \[
 \begin{split}
&\qquad f(x,y)=-3(x^2+y^2+4x+1)^2+4(2x+1)^3,\\
& P_1=(1,0),\quad
  P_2=\left(-\frac{1}{2},\frac{\sqrt{3}}{2}\right),\quad
  P_3=\left(-\frac{1}{2},-\frac{\sqrt{3}}{2}\right)
 \end{split}
 \]
 where $P_2,P_3$ are the  inner singularities of
 torus decomposition (V-1).
 Now we take following three polynomials
 $s_1$ $s_3$ and $s_5$ of degree $1$, $3$ and $5$ respectively:  
\[
\begin{split}
 s_1(x,y)&=x-Iy-1,\\
 s_3(x,y)&=\sqrt[3]{4}(3Iy^3-(5x+7)y^2-I(x-1)^2y-(x-1)^3),\\
 s_5(x,y)&=\sqrt{3}(y^5+3I(x+5)y^4-2
 (x^2+13x+10)y^3+2I(x-4)(x-1)^2y^2\\
&\hspace{5cm} -3(x+1)(x-1)^3y -I(x-1)^5).
 \end{split}
 \]
Then they satisfy the following equality:
\[
 s_1(x,y)^6f(x,y)=s_5(x,y)^2+s_3(x,y)^3.
\]
Note that $\{s_1=0\}$ does not pass through the inner singularities of $Q$.
Thus this decomposition is example  of
 $(2,3)$ quasi torus decomposition which does not  come from as in $(*_3)$.

\vspace{-0.3cm}
\begin{figure}[H]
 \begin{center}
  \includegraphics[scale=0.65]{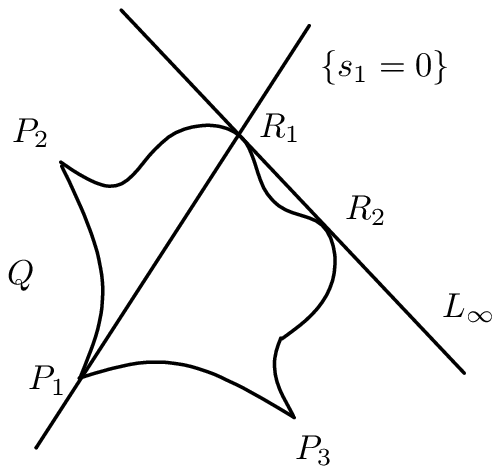}
 \end{center}
 \caption{}
 \label{a21}
\end{figure}
\vspace{-0.7cm}

\vspace{1cm}
\noindent
{\bf{Acknowledgment.}}
I would like to express our deepest gratitude to 
Professor Hiro-o Tokunaga  
who has proposed this problem.
I also express our deepest gratitude to 
Professor Mutsuo Oka for his various
advices during the preparation
of this paper.

\def\cprime{$'$} \def\cprime{$'$} \def\cprime{$'$} \def\cprime{$'$}
  \def\cprime{$'$}


\begin{thebibliography}{1}

\bibitem{BenoitTu}
B.~Audoubert, T.~C. Nguyen, and M.~Oka.
\newblock On {A}lexander polynomials of torus curves.
\newblock {\em J. Math. Soc. Japan}, 57(4):935--957, 2005.

\bibitem{Y}
Y.~Kenta.
\newblock On the topology of the complements of quartic and line
  configurations.
\newblock {\em SUT Journal of Mathematics}, 44:125--152, 2008.

\bibitem{Ku-Albanese}
V.~S. Kulikov.
\newblock On plane algebraic curves of positive {Albanese} dimension.
\newblock {\em Izv. Math}, 59(6):1173--1192, 1995.

\bibitem{Okadual}
M.~Oka.
\newblock Geometry of cuspidal sextics and their dual curves.
\newblock In {\em Singularities---Sapporo 1998}, pages 245--277. Kinokuniya,
  Tokyo, 2000.

\bibitem{OkaSurvey}
M.~Oka.
\newblock A survey on {Alexander} polynomials of plane curves.
\newblock {\em Singularit\'es Franco-Japonaise, S\'eminaire et congr\`es},
  10:209--232, 2005.

\bibitem{okatan}
M.~Oka.
\newblock Tangential {A}lexander polynomials and non-reduced degeneration.
\newblock In {\em Singularities in geometry and topology}, pages 669--704.
  World Sci. Publ., Hackensack, NJ, 2007.

\bibitem{Oka-Pho2}
M.~Oka and D.~Pho.
\newblock Classification of sextics of torus type.
\newblock {\em Tokyo J. Math. 25 (2002), no. 2}, pages 399--433, 2002.

\bibitem{Pho}
D.~T. Pho.
\newblock Classification of singularities on torus curves of type $(2,3)$.
\newblock {\em Kodai Math. J.}, 24(1):259--284, 2001.

\bibitem{Tokunaga-Kyoto}
H.-o. Tokunaga.
\newblock Dihedral covers and an elementary arithmetic on elliptic surfaces.
\newblock {\em J. Math. Kyoto Univ.}, 44(2):255--270, 2004.

\end{thebibliography}
\end{document}